\documentclass{article}
\usepackage{amsmath, amsfonts, amsthm, amssymb}
\usepackage{graphicx}
\usepackage{float}
\usepackage{verbatim}% for comment

\numberwithin{equation}{section}
\newtheorem{thm}{Theorem}[section]
\newtheorem{lma}[thm]{Lemma}
\newtheorem{cor}[thm]{Corollary}

\newtheorem{ques}[thm]{Question}

\renewcommand{\geq}{\geqslant}
\renewcommand{\leq}{\leqslant}

\title{Remarks on the analyticity of subadditive pressure for products of triangular matrices}
\author{Jonathan M. Fraser\\ \\
School of Mathematics, The University of Manchester,\\ Manchester, M13 9PL, UK\\
	E-mail: jonathan.fraser@manchester.ac.uk}

\begin{document}
\maketitle

\begin{abstract}
We study Falconer's subadditive pressure function with emphasis on analyticity.  We begin by deriving a simple closed form expression for the pressure in the case of diagonal matrices and, by identifying phase transitions with zeros of Dirichlet polynomials, use this to deduce that the pressure is piecewise real analytic.  We then specialise to the iterated function system setting and use a result of Falconer and Miao to extend our results to include the pressure for systems generated by matrices which are simultaneously triangularisable.  Our closed form expression for the pressure simplifies a similar expression given by Falconer and Miao by reducing the number of equations needing to be solved by an exponential factor.  Finally we present some examples where the pressure has a phase transition at a non-integer value and pose some open questions.
\end{abstract}

\section{Introduction}

Let $n \in \mathbb{N}$ and $\{A_i\}_{ i \in \mathcal{I}}$ be a finite collection of $n \times n$ non-singular matrices.  We define the \emph{subadditive pressure} for this system following Falconer \cite{affine}.  Let $\mathcal{I}^* = \bigcup_{k\geq1} \mathcal{I}^k$ denote the set of all finite sequences with entries in $\mathcal{I}$ and for
\[
\textbf{\emph{i}}= \big(i_1, i_2, \dots, i_k \big) \in \mathcal{I}^*
\]
write
\[
A_{\textbf{\emph{i}}} = A_{i_1} \circ A_{i_2} \circ \dots \circ A_{i_k}
\]
and  $\alpha_1(\textbf{\emph{i}}) \geq  \dots \geq \alpha_n(\textbf{\emph{i}})>0$ for the singular values of $A_\textbf{\emph{i}}$.  The \emph{singular values} of a linear map $A$ are the positive square roots of the eigenvalues of $A^T A$.  They are also the lengths of the semi-axes of the image of the unit ball under $A$ and thus correspond to how much $A$ contracts or expands in different directions.  For $s \in [0,n)$ the \emph{singular value function} $\phi^s: \mathcal{I}^* \to (0,\infty)$ is defined by
\[
\phi^s(\textbf{\emph{i}})  =\alpha_1(\textbf{\emph{i}}) \alpha_2(\textbf{\emph{i}})  \cdots \alpha_{m}(\textbf{\emph{i}}) \alpha_{m+1}(\textbf{\emph{i}})^{s-m}  
\]
where $m \in \{0, \dots, n-1\}$ is the unique non-negative integer satisfying $m \leq s < m+1$.  The singular value function leads us to define the \emph{pressure} $P : [0,n) \to \mathbb{R}$ corresponding to the system $\{A_i\}_{ i \in \mathcal{I}}$ by
\[
P(s) = \lim_{k \to \infty} \frac{1}{k} \log \sum_{\textbf{\emph{i}} \in \mathcal{I}^k} \phi^s(\textbf{\emph{i}})
\]
where the limit exists since the singular value function is submultiplicative in $\textbf{\emph{i}}$, i.e.
\[
\phi^s(\textbf{\emph{i}}\, \textbf{\emph{j}}) \leq \phi^s(\textbf{\emph{i}}) \, \phi^s(\textbf{\emph{j}})
\]
for all $\textbf{\emph{i}}, \, \textbf{\emph{j}} \in \mathcal{I}^*$, see \cite[Lemma 2.1]{affine}.  It is convenient to extend the domain of $P$ to $[0,\infty)$ and so we let
\[
P(s) =  \log \sum_{i \in \mathcal{I}} \det(A_i)^{s/n}
\]
for $s \geq n$.  Here the pressure is defined without the need for a limit as the determinant is multiplicative.  It is easy to see that $P$ is continuous on $[0,\infty)$ and convex on each interval $(m,m+1)$, with $m \in \{0,\dots, n-1\}$, and on $(n, \infty)$. Moreover, it is easy to contruct examples where the pressure is not convex on an interval containing an integer; see Section \ref{examplessection}.  It is a simple consequence of piecewise convexity that $P$ is differentiable at all but at most countably many points and semi-differentiable everywhere.  The main focus of this article is to study real analyticity of the pressure and our main application is that the pressure is always piecewise real analytic for products of matrices which are simultaneously triangularisable, see Corollaries \ref{mainanalytic} and \ref{mainanalytictri}.  Moreover, the number of phase transitions, and therefore points where the pressure is not smooth, can be bounded in terms of the spatial dimension and the number of matrices.  We also provide examples showing that the pressure can have phase transitions at non-integer values.  Phase transitions in the interval $(0,1)$ have previously been exhibited by K\"aenm\"aki and Vilppolainen \cite[Example 6.5]{kaenmakisub}.
\\ \\
We say a real valued function on some domain $D \subseteq \mathbb{R}$ is \emph{piecewise real analytic} if $D$ can be written as the closure of the union of a finite collection of open (possibly unbounded) intervals with the function being real analytic on each interval.  The boundary points of the open intervals which are in the interior of $D$ are called \emph{phase transitions}, provided that the function is not real analytic on any neighbourhood of the point.  Note that if a piecewise real analytic function is continuous, then it is completely defined by its values on the open intervals where it is real analytic.
\\ \\
One of the main applications of the subadditive pressure function discussed in this paper is in the study of self-affine fractals.  In particular, if the matrices $\{A_i\}_{i \in \mathcal{I}}$ are chosen to be contractions and to each matrix we associate a translation vector $t_i \in \mathbb{R}^n$, then we have an iterated function system $\{A_i+t_i\}_{i \in \mathcal{I}}$, which has a unique non-empty compact attractor $F$, called the self-affine set for the system.  Alternatively, assuming some separation conditions, one can view $F$ as the repeller of a uniformly expanding map defined by the inverse branches of the contraction mappings.  In either case, the pressure is related to many interesting geometric properties of $F$ and the associated dynamical system.  Perhaps most notably the unique zero of the pressure gives an upper bound for the Hausdorff dimension of $F$ and a `best guess' for the actual Hausdorff dimension.  These ideas date back to Douady-Oesterl\'e \cite{douady} and Falconer \cite{affine, affine2}.  In \cite{affine} Falconer proved that the zero of the pressure gives the Hausdorff dimension of $F$ for Lebesgue almost all choices of $\{t_i\}_{i \in \mathcal{I}}$ provided the matrices all have singular values strictly less than $1/3$, which was relaxed to the optimal constant $1/2$ by Solomyak \cite{solomyak}.  Since then the subadditive pressure, and several related functions, have received a lot of attention in the literature on self-affine fractals and non-conformal dynamics.  There have also been several extensions of these ideas to nonlinear systems, see Falconer \cite{falconerrepeller} and Barreira \cite{barreira}.  Due to their focus on upper triangular systems, the papers of Falconer-Miao \cite{miao}, Falconer-Lammering \cite{lammering}, Manning-Simon \cite{manning} and B\'ar\'any \cite{barany} are particularly relevant to our study.
\\ \\
It is worth remarking that \emph{additive pressure functions} associated to uniformly hyperbolic dynamical systems and self-conformal fractals were studied before the more complicated subadditive analogues, see \cite{bowen, bowen2, ruelle}.  The additive setting is rather simpler and if the associated potential is taken to be the appropriate analogue of the singular value function, then the pressure is real analytic on its whole domain.  This is a special case of a more general result of Ruelle \cite{ruelle}.  The proof relies on a transfer operator approach, which does not apply in the non-conformal (or self-affine) setting.
\\ \\
One of the reasons the analyticity (or differentiability) of the pressure is interesting is that it is related to the number of ergodic equilibrium measures for the pressure (this was drawn to our attention by Pablo Shmerkin).  Indeed such links have been investigated by Feng-K\"aenm\"aki \cite{fengkaenmaki} and Guivarc'h-Le Page \cite{lepage}, albeit in a slightly different context.

\section{Results}

\subsection{Subadditive pressure for diagonal matrices}

Suppose the matrices $\{A_i\}_{ i \in \mathcal{I}}$ are all diagonal and write $c_1(i), \dots, c_n(i)>0$ for the absolute values of the diagonal entries of $A_i$.  Note that the sets $\{c_1, \dots, c_n\}$ and $\{\alpha_1(i), \dots, \alpha_n(i)\}$ are equal but one cannot say anything about the relative ordering.  Indeed, once one starts composing diagonal matrices, the order in which the singular values appear down the main diagonal of the matrix can change, which is one of the main difficulties in computing the pressure.  For $\textbf{\emph{i}} = (i_1, i_2, \dots, i_k) \in \mathcal{I}^*$ write $c_1(\textbf{\emph{i}}), \dots, c_n(\textbf{\emph{i}})$ for the diagonal entries of $A_\textbf{\emph{i}}$, noting that
\[
 c_l(\textbf{\emph{i}}) = c_l(i_1) \cdots c_l(i_k)
\]
for each $l \in \{1, \dots, n\}$.  Let $S_n$ be the symmetric group on $\{1, \dots, n\}$ and for each $\sigma \in S_n$ and $s \in [0,n)$ we define the \emph{$\sigma$-ordered singular value function} $\phi^s_\sigma:\mathcal{I}^* \to (0,\infty)$ by
\[
\phi^s_\sigma(\textbf{\emph{i}})  =c_{\sigma(1)} (\textbf{\emph{i}}) c_{\sigma(2)}(\textbf{\emph{i}})  \cdots c_{\sigma(m)}(\textbf{\emph{i}}) c_{\sigma(m+1)}(\textbf{\emph{i}})^{s-m}  
\]
where $m \in \{ 0, \dots, n-1\}$ is the unique non-negative integer satisfying $m \leq s < m+1$.  The key advantage of these ordered singular value functions is that they are multiplicative in $\textbf{\emph{i}}$ instead of only submultiplicative, i.e.
\[
\phi_\sigma^s(\textbf{\emph{i}}\, \textbf{\emph{j}}) = \phi_\sigma^s(\textbf{\emph{i}}) \, \phi_\sigma^s(\textbf{\emph{j}})
\]
for all $\textbf{\emph{i}}, \, \textbf{\emph{j}} \in \mathcal{I}^*$ and $\sigma \in S_n$.  This allows us to define the associated pressure by means of a closed form expression, without taking a limit.  More precisely, we define the \emph{$\sigma$-ordered pressure} $P_\sigma : [0,n) \to \mathbb{R}$ by
\[
P_\sigma(s) = \log \sum_{i \in \mathcal{I}} \phi_\sigma^s(i)
\]
and observe that
\[
\sum_{\textbf{\emph{i}} \in \mathcal{I}^k} \phi_\sigma^s(\textbf{\emph{i}}) = \Bigg(\sum_{i \in \mathcal{I}} \phi_\sigma^s(i) \bigg)^k
\]
for all $k \in \mathbb{N}$. We extend the domain of each $P_\sigma$ to $[0,\infty)$ as before by setting $P_\sigma(s) = P(s)$ for $s \geq n$, since the ordering of the diagonal entries of a diagonal matrix does not change the determinant.  Again, it is easy to see that $P_\sigma$ is continuous on $[0,\infty)$ and convex on each interval $(m,m+1)$, with $m \in \{0,\dots, n-1\}$, and on $(n, \infty)$.  Moreover, it is immediate that $P_\sigma$ is piecewise real analytic, with the only possible phase transitions occurring at the points $\{1, \dots, n\}$.
\begin{thm} \label{mainmax}
For all $s \in [0, \infty)$ we have
\[
P(s) = \max_{\sigma \in S_n} P_\sigma(s).
\]
\end{thm}
We will prove Theorem \ref{mainmax} in Section \ref{mainmaxproof}.  In the case of $2 \times 2$ matrices, where the pressure is the maximum of two functions, this can be found in \cite{manning}.  In fact \cite{manning} dealt with certain nonlinear maps corresponding to upper triangular matrices.    The key point of this result is that we have a closed form expression for the pressure, which is very useful for computational purposes and for analysing differentiability and analyticity, since differentiating a function defined by a limit is awkward.  First and foremost, by identifying phase transitions in the pressure with zeros of Dirichlet polynomials, we can deduce the following result.
\begin{cor} \label{mainanalytic}
For products of non-singular diagonal matrices, the pressure is piecewise real analytic.
\end{cor}
We will prove Corollary \ref{mainanalytic} in Section \ref{mainanalyticproof}. We are able to bound the number of phase transitions (and therefore the number of `pieces' in the piecewise decomposition of $P$) in terms of the number of matrices $\lvert \mathcal{I} \rvert$ and the spatial dimension $n$, however we defer discussion of the explicit bound until Sections \ref{mainanalyticproof} and \ref{questions}.  It is now possible to give various sufficient conditions for $P$ to be real analytic on the whole interval $(m,m+1)$, however, we refrain from stating a myriad of different examples because in practice one would simply plot the different ordered pressures and observe which is the maximum.  Then on any interval where one ordered pressure is bigger than or equal to all the others, $P$ is real analytic.  However, we do state one sufficiency result which we find particularly intuitive.
\begin{cor} \label{mainanal}
Let $m \in \{0, \dots, n-1\}$.  If there exists $\sigma \in S_n$ such that for all $i \in \mathcal{I}$
\[
\{\alpha_1(i), \dots, \alpha_{m}(i)\} = \{ c_{\sigma(1)}(i), \dots, c_{\sigma(m)}(i)\}
\]
and
\[
\alpha_{m+1}(i) =  c_{\sigma(m+1)}(i),
\]
then
\[
P(s) = P_\sigma(s)
\]
for all $s \in [m,m+1]$ and, in particular, the pressure is real analytic on $(m,m+1)$.
\end{cor}
We will prove Corollary \ref{mainanal} in Section \ref{mainanalproof}.  Notice that (especially for large $m$) the sufficient condition for analyticity given above is weaker than requiring $\alpha_1(i) = c_{\sigma(1)}(i)$, $\dots$, $\alpha_{m}(i) = c_{\sigma(m)}(i)$ and $\alpha_{m+1}(i) =  c_{\sigma(m+1)}(i)$.  However, in that more restrictive setting, we get the following precise corollary.
\begin{cor}
If there exists $\sigma \in S_n$ such that for all $i \in \mathcal{I}$ and $l \in \{1, \dots, n\}$
\[
\alpha_l(i) = c_{\sigma(l)}(i),
\]
then
\[
P(s) = P_\sigma(s)
\]
for all $s \in [0, \infty)$ and, in particular, the pressure is real analytic on each interval $(m,m+1)$ with $m \in \{0, \dots, n-1\}$.
\end{cor}

In light of Theorem \ref{mainmax}, non-trivial phase transitions, i.e., phase transitions occurring at non-integer values, can only happen at points when the maximum of the ordered pressures `changes hands' between two different ordered pressures.  It is not immediately obvious that this is possible, but it does not take long to find such examples.  We will present some examples of non-trivial phase transitions in Section \ref{examplessection}, as well as a simple example where Corollary \ref{mainanal} can be applied to certain intervals.
\\ \\
We conclude this section with the combinatorial observation that, despite there being $n!$ different ordered pressures, there are significantly fewer distinct ones.  In particular, we choose the first $m$ entries in the ordered singular value functions, with the ordering irrelevant, and then choose the $(m+1)$th entry from the remaining $n-m$ choices. As such, if we are interested in analysing the pressure in the interval $[m,m+1)$, for some $m \in \{0, \dots, n-1\}$, then we have to take the maximum of
\begin{equation} \label{numberof}
\left( \begin{array}{c}
n \\
m
\end{array} \right) \cdot \left( \begin{array}{c}
n -m \\
1
\end{array} \right) \  = \  n \left( \begin{array}{c}
n -1 \\
m
\end{array} \right)
\end{equation}
(possibly) distinct functions.

\subsection{Self-affine sets generated by simultaneously triangularisable matrices}

In this section assume that $\{A_i\}_{ i \in \mathcal{I}}$ are all contracting upper triangular matrices and as before write $c_1(i), \dots, c_n(i) \in (0,1)$ for the absolute values of the diagonal entries of $A_i$.  An interesting, and perhaps surprising result, of Falconer and Miao \cite{miao} is that the pressure in this setting only depends on the diagonal entries.  Moreover, they gave a closed form expression for the pressure in the interval $[m,m+1)$ for $m \in \{0, \dots, n-1\}$ as the maximum of functions of the form
\[
\log \ \sum_{ i \in \mathcal{I}} \big(c_{j_1}(i) \cdots c_{j_{m}}(i)  \big)^{m+1-s}  \big(c_{j'_1}(i) \cdots c_{j'_{m+1}}(i)  \big)^{s-m} 
\]
over all independent choices of subsets $\{j_1, \dots , j_{m}\}$ and $\{j'_1, \dots , j'_{m+1}\}$ of $\{1,\dots, n\}$, see \cite[Theorem 2.5]{miao}.  In particular, in the interval $[m,m+1)$, one takes the maximum of
\[
\left( \begin{array}{c}
n \\
m
\end{array} \right) \cdot \left( \begin{array}{c}
n \\
m+1
\end{array} \right)
\]
functions.  For related results see \cite{barany, lammering, manning}.  Since the pressure does not depend on the non-diagonal entries of the matrices, we can apply Theorem \ref{mainmax} also in the upper triangular setting, simply by ignoring the non-diagonal entries.  As such and in view of (\ref{numberof}) we can reduce the number of functions needed in the interval $[m,m+1)$ by a factor of
\[
 \left( \begin{array}{c}
n \\
m+1
\end{array} \right) /  \left( \begin{array}{c}
n-m \\
1
\end{array} \right)
\]
which grows exponentially in $n$ in the central intervals.  More precisely, applying Stirling's formula, the above factor is larger than $2^{n}/(n\sqrt{2n})$ for $n \geq 2$ and choosing $m$ to be the integer part of $n/2$.  The rest of the results in the previous section carry over to the upper triangular case, or indeed any set of matrices which are simultaneously triangularisable, i.e. there exists a basis with respect to which all of the matrices are either upper or lower triangular.  Most notably we have the following general result.
\begin{cor} \label{mainanalytictri}
For products of contracting non-singular simultaneously triangularisable matrices, the pressure is piecewise real analytic.
\end{cor}

\section{Examples} \label{examplessection}

Let $n=3$ and let $T_1$ and $T_2$ be $3 \times 3$ upper triangular matrices with non-zero positive diagonal entries $c_1(1), c_2(1), c_3(1)$ and  $c_1(2), c_2(2), c_3(2)$ respectively.  Theorem \ref{mainmax} and (\ref{numberof}) show that the pressure corresponding to this system is given by the maximum of three functions in the interval $[0,1)$, six functions in the interval $[1,2)$ and three functions in the interval $[2,3)$.  By choosing the diagonal entries appropriately, we can create a phase transition in each of these intervals.  Choosing
\[
c_1(1) = 0.9, c_2(1)=0.4 , c_3(1) =0.6, c_1(2)=0.1 , c_2(2)=0.4, c_3(2)=0.2 
\]
gives the pressure a phase transition at the point $s_1 = 0.5 \in (0,1)$ with $P_-'(s_1)  \approx -0.916 < -0.655 \approx P_+'(s_1)$. Choosing
\[
c_1(1) = 0.1, c_2(1)=0.2 , c_3(1) =0.9, c_1(2)=0.9 , c_2(2)=0.4, c_3(2)=0.2 
\]
gives the pressure a phase transition at a point $s_2 \approx 1.193 \in (1,2)$ with $P_-'(s_2)  \approx -1.469 < -0.978 \approx P_+'(s_2)$.  Finally, choosing
\[
c_1(1) = 0.9, c_2(1)=0.5 , c_3(1) =0.8, c_1(2)=0.9 , c_2(2)=0.5, c_3(2)=0.01 
\]
gives the pressure a phase transition at a point $s_3 \approx 2.156 \in (2,3)$ with $P_-'(s_3) \approx -1.695 < -0.693  \approx P_+'(s_3)$. 

\begin{figure}[H]
	\centering
	\includegraphics[width=162mm]{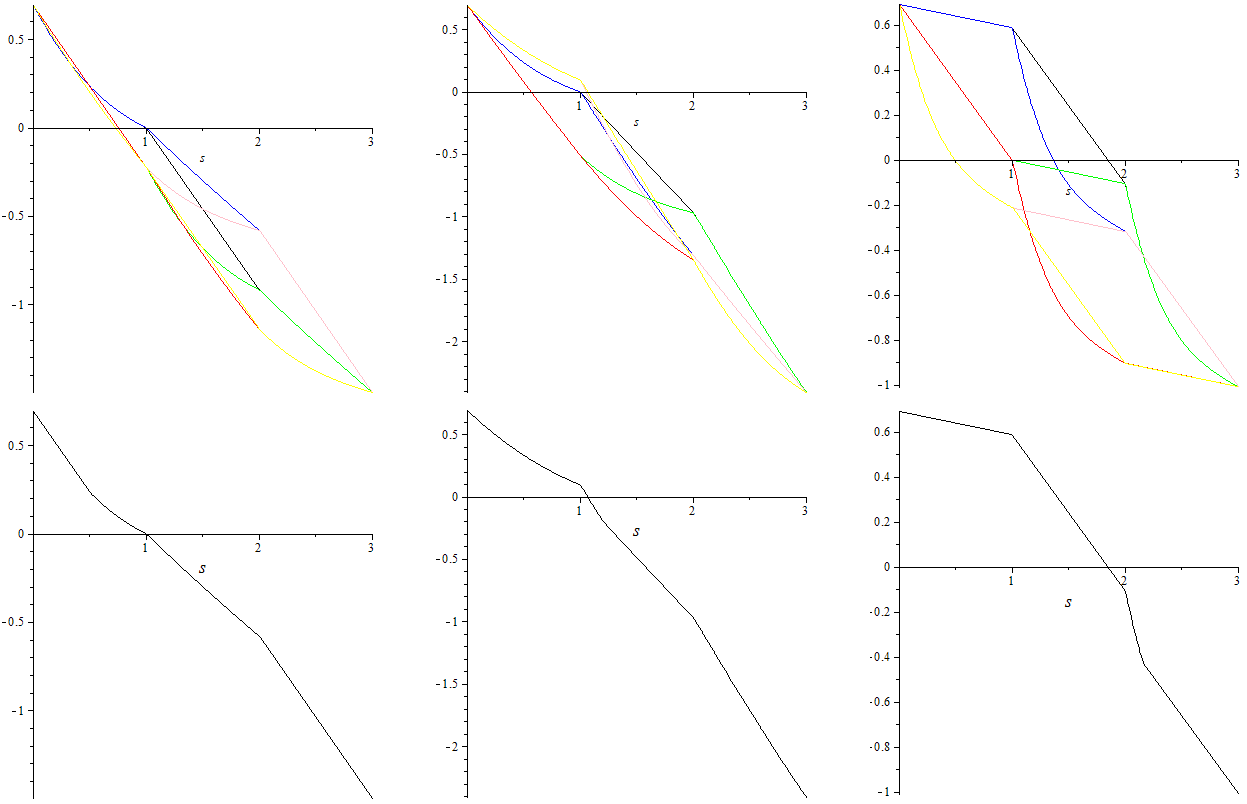}
\caption{Top row: plots of the ordered pressures in the range $[0,3]$ for each of the three examples described above.  The permutations (written as cycles) corresponding to each colour are as follows: black: (1), blue: (23), green: (12), red: (132), pink: (123), yellow: (13).  Bottom row: plots of the standard pressure, which is equal to the maximum of the ordered pressures.}
\end{figure}

For our second example, let $n=7$ and let $T_1$ and $T_2$ be given by
\[
T_1 = \left( \begin{array}{ccccccc}
2 & -6 & 15 & 0 & -2 & 0 & 2 \\
0 &- 1 & 0 & 1 & -6 & 0 & 0 \\
0 & 0 & 10 & 4 & 9 & 6 & 0 \\
0 & 0 & 0 & 8 & -2 & 0 & 1 \\
0 & 0 & 0 & 0 & -5 & -3 & 4 \\
0 & 0 & 0 & 0 & 0 & 7 & 7 \\
0 & 0 & 0 & 0 & 0 & 0 & 4 \\
\end{array} \right)  \hspace{5mm} \text{and} \hspace{5mm} T_2 = \left( \begin{array}{ccccccc}
3 & 2 & 5 & 0 & -6 & -4 & 2 \\
0 & 1 & 2 & 8 & 6 & 1 & 6 \\
0 & 0 & -14 & 1 & 1 & 13 & 3 \\
0 & 0 & 0 & 11 & 9 & 0 & 9 \\
0 & 0 & 0 & 0 & 4 & 10 & 1 \\
0 & 0 & 0 & 0 & 0 & -15 & -5 \\
0 & 0 & 0 & 0 & 0 & 0 & 2\\
\end{array} \right)
\]
Choosing
\[
\sigma =  \left( \begin{array}{ccccccc}
1 & 2 & 3 & 4 & 5 & 6 & 7 \\
3 & 4 & 6 & 5 & 7 & 1 & 2 \\
\end{array} \right) 
\]
we can apply Corollary \ref{mainanal} in the intervals $(3,4)$ and $(6,7)$ to deduce that the pressure $P(s) = P_\sigma(s)$, and is hence real analytic, in these regions.  Of course we could have just plotted all of the ordered pressures and deduced the regions where the maximum was real analytic, however that would involve plotting 140 functions in the interval $(3,4)$, for example.

\newpage

\section{Proofs}

\subsection{Proof of Theorem \ref{mainmax}} \label{mainmaxproof}

\begin{lma} \label{lemma1}
For all $s \in [0,n)$ and $\textbf{i} \in \mathcal{I}^*$, we have $\phi^s(\textbf{i}) =  \max_{\sigma \in S_n} \phi^s_\sigma(\textbf{i})$.
\end{lma}

\begin{proof}
Let $\textbf{\emph{i}} \in \mathcal{I}^*$ and suppose $s \in  [m,m+1)$ for some $m \in \{0,\dots, n-1\}$. Clearly $\phi^s(\textbf{\emph{i}})$ is equal to $\phi^s_\sigma(\textbf{\emph{i}})$ for some $\sigma$ and so $\phi^s(\textbf{\emph{i}}) \leq  \max_{\sigma \in S_n} \phi^s_\sigma(\textbf{\emph{i}})$.  Also, in trying to maximise  $\phi^s_\sigma(\textbf{\emph{i}})$ over $\sigma$, one must choose a permutation for which
\[
\{\alpha_1(\textbf{\emph{i}}), \dots, \alpha_{m+1}(\textbf{\emph{i}})\} = \{ c_{\sigma(1)}(\textbf{\emph{i}}), \dots, c_{\sigma(m+1)}(\textbf{\emph{i}})\},
\]
i.e. a permutation which `uses' the largest $(m+1)$ singular values and excludes the other (smaller) values.  Fix such a permutation $\sigma$.  Since $\phi^s_\sigma(\textbf{\emph{i}})$ is symmetric in the values $ c_{\sigma(1)}(\textbf{\emph{i}}), \dots , c_{\sigma(m)}(\textbf{\emph{i}})$, the ordering of the first $m$ terms is irrelevant, and so the only question is which singular value to choose as $c_{\sigma(m+1)}(\textbf{\emph{i}})$.  Suppose $\sigma$ is such that $c_{\sigma(m+1)}(\textbf{\emph{i}}) \neq \alpha_{m+1}(\textbf{\emph{i}})$.  Cancelling common terms we have
\[
\frac{\phi^s(\textbf{\emph{i}})}{\phi^s_\sigma(\textbf{\emph{i}})} \ = \ \frac{c_{\sigma(m+1)}(\textbf{\emph{i}}) \, \alpha_{m+1}(\textbf{\emph{i}})^{s-m} }{ \alpha_{m+1}(\textbf{\emph{i}}) \, c_{\sigma(m+1)}(\textbf{\emph{i}})^{s-m} }  \ = \  \bigg( \frac{ c_{\sigma(m+1)}(\textbf{\emph{i}}) }{ \alpha_{m+1}(\textbf{\emph{i}}) }\bigg)^{m+1-s} \ \geq  \ 1
\]
since $c_{\sigma(m+1)}(\textbf{\emph{i}}) \geq \alpha_{m+1}(\textbf{\emph{i}})$ and $m+1-s > 0$, which gives $\phi^s(\textbf{\emph{i}}) \geq  \max_{\sigma \in S_n} \phi^s_\sigma(\textbf{\emph{i}})$ and completes the proof.
\end{proof}

\begin{lma} \label{lemma2}
For all $s \in [0,n)$, we have
\[
\Bigg( \max_{\sigma \in S_n} \ \sum_{\textbf{i} \in \mathcal{I}} \phi^s_\sigma(\textbf{i}) \Bigg)^k  \ \leq \ \sum_{\textbf{i} \in \mathcal{I}^k} \phi^s(\textbf{i})  \ \leq \ n! \ \Bigg(  \max_{\sigma \in S_n} \ \sum_{\textbf{i} \in \mathcal{I}} \phi^s_\sigma(\textbf{i}) \Bigg)^k .
\]
\end{lma}

\begin{proof}
Observe that by Lemma \ref{lemma1}
\[
\sum_{\textbf{\emph{i}} \in \mathcal{I}^k} \phi^s(\textbf{\emph{i}}) \  = \ \sum_{\textbf{\emph{i}} \in \mathcal{I}^k}\max_{\sigma \in S_n} \phi_\sigma^s(\textbf{\emph{i}}) \ \geq \ \max_{\sigma \in S_n} \ \sum_{\textbf{\emph{i}} \in \mathcal{I}^k} \phi^s_\sigma(\textbf{\emph{i}}) \ = \  \Bigg(  \max_{\sigma \in S_n} \ \sum_{\textbf{\emph{i}} \in \mathcal{I}} \phi^s_\sigma(\textbf{\emph{i}}) \Bigg)^k
\]
since the ordered singular value functions are multiplicative.  This yields the left hand inequality in the statement of the lemma.  To obtain the right hand inequality, observe that $\phi^s(\textbf{\emph{i}}) = \phi^s_\sigma(\textbf{\emph{i}})$ for some $\sigma$ and so
\[
\sum_{\textbf{\emph{i}} \in \mathcal{I}^k} \phi^s(\textbf{\emph{i}}) \ \leq  \ \sum_{\textbf{\emph{i}} \in \mathcal{I}^k} \ \sum_{\sigma \in S_n} \phi^s_\sigma(\textbf{\emph{i}}) \ = \   \sum_{\sigma \in S_n} \ \Bigg( \sum_{\textbf{\emph{i}} \in \mathcal{I}} \phi^s_\sigma(\textbf{\emph{i}}) \Bigg)^k \ \leq \ n! \ \Bigg(  \max_{\sigma \in S_n} \ \sum_{\textbf{\emph{i}} \in \mathcal{I}} \phi^s_\sigma(\textbf{\emph{i}}) \Bigg)^k 
\]
again using multiplicativity of the ordered singular value functions.
\end{proof}

Theorem \ref{mainmax} now follows easily by applying Lemma \ref{lemma2} to obtain
\[
P(s) \ \geq \  \log \Bigg( \max_{\sigma \in S_n} \ \sum_{\textbf{\emph{i}} \in \mathcal{I}} \phi^s_\sigma(\textbf{\emph{i}}) \Bigg) \ = \   \max_{\sigma \in S_n} \ \log \sum_{\textbf{\emph{i}} \in \mathcal{I}} \phi^s_\sigma(\textbf{\emph{i}})  \ = \  \max_{\sigma \in S_n} \ P_\sigma(s)
\]
and
\[
P(s) \ \leq \  \lim_{k \to \infty} \frac{1}{k} \log n! \ + \  \log \Bigg( \max_{\sigma \in S_n} \ \sum_{\textbf{\emph{i}} \in \mathcal{I}} \phi^s_\sigma(\textbf{\emph{i}}) \Bigg) \ = \   \max_{\sigma \in S_n} \ \log \sum_{\textbf{\emph{i}} \in \mathcal{I}} \phi^s_\sigma(\textbf{\emph{i}})  \ = \  \max_{\sigma \in S_n} \ P_\sigma(s).
\]

\subsection{Proof of Corollary \ref{mainanalytic}} \label{mainanalyticproof}

To prove that $P$ is piecewise real analytic it suffices to show that for a given $m \in \{0, \dots, n-1\}$ and two given permutations $\sigma, \tau \in S_n$, if the ordered pressures $P_\sigma$ and $P_\tau$ are not equal on the entire interval $(m,m+1)$, then their graphs can only intersect a finite number of times. This is equivalent to showing that the function
\begin{eqnarray*}
E(s) &:=& \sum_{i \in \mathcal{I}} \phi_\sigma^s(\textbf{\emph{i}})  - \sum_{i \in \mathcal{I}} \phi_\tau^s(\textbf{\emph{i}})  \\ \\
&=& \sum_{i \in \mathcal{I}} \frac{c_{\sigma(1)} (\textbf{\emph{i}}) c_{\sigma(2)}(\textbf{\emph{i}})  \cdots c_{\sigma(m)}(\textbf{\emph{i}})}{c_{\sigma(m+1)}(\textbf{\emph{i}})^{m}} c_{\sigma(m+1)}(\textbf{\emph{i}})^{s} -   \frac{c_{\tau(1)} (\textbf{\emph{i}}) c_{\tau(2)}(\textbf{\emph{i}})  \cdots c_{\tau(m)}(\textbf{\emph{i}})}{c_{\tau(m+1)}(\textbf{\emph{i}})^{m}} c_{\tau(m+1)}(\textbf{\emph{i}})^{s}
\end{eqnarray*}
has at most finitely many zeros in the interval $(m,m+1)$, assuming it is not identically zero.  However, this is quickly seen to be true since $E(s)$ is a (generalised) Dirichlet polynomial and therefore can have at most $2 \lvert \mathcal{I} \rvert - 1$ zeros in $\mathbb{R}$.  Recall that Dirichlet polynomials are functions of the form
\[
\sum_{i = 1}^N a_i b_i^s
\]
with $a_i \in \mathbb{R}$ and $b_i>0$.  A classical result, which can be proved by applying Rolle's Theorem, is that such functions have at most $N-1$ zeros, provided they are not identically zero.  For further information on zeros of Dirichlet polynomials and related topics, see Jameson \cite{jameson}.
\\ \\
If we are interested in bounding the number of phase transitions explicitly, then the following crude estimate can be deduced.  We can have trivial phase transitions at the points $\{1, \dots, n\}$.  For non-trivial phase transitions in the interval $(m,m+1)$ for $m \in \{0, \dots, n-1\}$, we know that each distinct pair of ordered pressures can give rise to at most $2 \lvert \mathcal{I} \rvert - 1$ phase transitions by the above argument and using (\ref{numberof}) there are at most
\[
\left( \begin{array}{c}
n \left( \begin{array}{c}
n -1 \\
m
\end{array} \right)\\
2
\end{array} \right)  % \ = \ \frac{ n!}{2(n-m-1)! \, m!} \bigg( \frac{ n!}{(n-m-1)! \, m!}-1 \bigg)
\]
distinct pairs of ordered pressures.  This yields the following upper bound for the total number of phase transitions:
\[
n \ + \  \big( 2 \lvert \mathcal{I} \rvert -1 \big)\ \sum_{m=0}^{n-1}\left( \begin{array}{c}
n \left( \begin{array}{c}
n -1 \\
m
\end{array} \right)\\
2
\end{array} \right).
\]
We can simplify the summation as follows:
\begin{eqnarray*}
\sum_{m=0}^{n-1}\left( \begin{array}{c}
n \left( \begin{array}{c}
n -1 \\
m
\end{array} \right)\\
2
\end{array} \right)& = & \frac{n^2}{2} \sum_{m=0}^{n-1} \left( \begin{array}{c}
n-1\\
m
\end{array} \right)^2  \ - \ \frac{n}{2} \sum_{m=0}^{n-1} \left( \begin{array}{c}
n-1\\
m
\end{array} \right)\\
\\ \\
& = & \frac{n^2}{2} \left( \begin{array}{c}
2n-2\\
n-1
\end{array} \right)  \ - \ \frac{2^nn}{4} 
\\ \\
 & = &  \frac{n^3}{8n-4} \left( \begin{array}{c}
2n\\
n
\end{array} \right)-\frac{2^nn}{4} \\ \\
 & \sim &  \frac{n \sqrt{n} \,  4^n}{8\sqrt{\pi}} 
\end{eqnarray*}
as $n \to \infty$, where the final line giving the asymptotic value was obtained by applying Stirling's formula to the binomial coefficient.

\subsection{Proof of Corollary \ref{mainanal}} \label{mainanalproof}

Let $m \in \{0, \dots, n-1\}$ and suppose $\sigma \in S_n$ is such that for all $i \in \mathcal{I}$
\[
\{\alpha_1(i), \dots, \alpha_{m}(i)\} = \{ c_{\sigma(1)}(i), \dots, c_{\sigma(m)}(i)\}
\]
and
\[
\alpha_{m+1}(i) =  c_{\sigma(m+1)}(i).
\]
By following the proof of Lemma \ref{lemma1}, it is easily seen that $\phi_\sigma^s(i) = \max_{\sigma' \in S_n} \phi_{\sigma'}^s(i)$ for all $i \in \mathcal{I}$ and $s \in [m,m+1]$, and therefore by Theorem \ref{mainmax}
\[
P(s) = \max_{\sigma' \in S_n} P_{\sigma'}(s) = P_\sigma(s)
\]
for all $s \in [m,m+1]$, completing the proof. \hfill \qed

\section{Some open questions and discussion} \label{questions}

We have proved that the pressure is piecewise real analytic for products of diagonal matrices and simultaneously triangularisable matrices.  However, this falls significantly short of proving this in general and we therefore ask the following question.

\begin{ques}
Is the pressure always piecewise real analytic or at least piecewise differentiable?
\end{ques}

In our setting we can bound the number of phase transitions by
\begin{equation} \label{boundd}
n \ + \   \big( 2 \lvert \mathcal{I} \rvert -1 \big) \ \Bigg(\frac{n^3}{8n-4} \left( \begin{array}{c}
2n\\
n
\end{array} \right)-\frac{2^nn}{4}\Bigg),
\end{equation}
however, this is very crude.  For a fixed spatial dimension, (\ref{boundd}) grows linearly in the number of matrices, which seems reasonable, but for a fixed number of matrices it grows as
\[
  \sim \   \frac{ 2 \lvert \mathcal{I} \rvert -1 }{8\sqrt{\pi}} \, n \sqrt{n} \,  4^n
\]
as the spatial dimension $n \to \infty$, which seems far too fast and gives poor estimates.  For example, for 2 matrices in dimension 5 the explicit bound is 2510.  It would be interesting to search for optimal bounds or to just improve (\ref{boundd}).
\begin{ques}
In the setting of upper triangular matrices, what is the optimal bound on the number of phase transitions for the pressure in terms of $\lvert \mathcal{I} \rvert$ and $n$?
\end{ques}
It would certainly be possible to reduce the bound (\ref{boundd}) via a more careful application of Rolle's Theorem or Descartes' rule of signs, to the Dirichlet polynomial $E(s)$, but we omit further details.  We emphasise that the purpose of this paper is to prove piecewise analyticity and not to study combinatorial issues concerning the sharpness of the bound on the possible number of phase transitions.  Another possible problem to consider is the existence and nature of \emph{higher order phase transitions}, i.e. points for which the pressure is $C^k$ but not $C^{k+1}$ for some $k$.  We have only been able to exhibit 0th order phase transitions, i.e. points where the pressure is continuous but not differentiable.  Since our main result gives an explicit formula for the pressure, it should provide a useful tool in searching for higher order phase transitions, but we have not pursed this here.  Finally, we ask a more open ended question.
\begin{ques}
Is there any interesting geometric or dynamical significance of the ordered pressures in regions where they are strictly less than the subadditive pressure?
\end{ques}

\vspace{9mm}

\begin{centering}

\textbf{Acknowledgements}

This work was completed while the author was a Research Fellow at the University of Warwick where he was financially supported by the EPSRC grant EP/J013560/1.  He thanks Pablo Shmerkin for helpful discussions and for providing some useful references.
\end{centering}

\end{document}